\documentclass[12pt]{amsart}
\usepackage[active]{srcltx}
\usepackage{amsmath,amssymb}
\usepackage{latexsym}
\usepackage{color}
\usepackage{dsfont}
\usepackage[normalem]{ulem}
\usepackage{epsfig}
\usepackage{pgf,pgfarrows,pgfnodes,pgfautomata,pgfheaps}
\usepackage{verbatim}
\usepackage{enumerate}
\usepackage{url}
\newtheorem{teo}{Theorem}[section]
\newtheorem{Theorem}{Theorem}[section]
\newtheorem{lema}[teo]{Lemma}
\newtheorem{prop}[teo]{Proposition}
\newtheorem{defi}[teo]{Definition}
\newtheorem{cor}[teo]{Corollary}
\newtheorem{rem}[teo]{Remark}

\usepackage{epstopdf}

\begin{document}

\title[Quenching phenomena with extra terms]{\bf A non-local quenching system with coupled and uncoupled singular absorption terms.}

\author[Sergio Junquera]{Sergio Junquera$^1$}
    \address{Sergio Junquera
    \hfill\break\indent  Departamento de An\'alisis Matem\'atico y Matem\'{a}tica Aplicada,
    \hfill\break\indent Universidad Complutense de Madrid, 28040 Madrid, Spain.
    \hfill\break\indent  e-mail: {\tt sejunque@ucm.es}}
    \thanks{$^1$ Corresponding author: {\tt sejunque@ucm.es}}

\maketitle

\begin{abstract}
In this paper we study a non-local diffusion system of two equations with both coupled and uncoupled singular absorption terms of the type $u^{-p}$. We prove that there exist necessary and sufficient conditions for the existence of both stationary and quenching solutions. We also characterize in terms of the exponents of the absorption terms when the quenching is simultaneous or non-simultaneous and obtain the quenching rates.

\end{abstract}

\begin{section}{Introduction}

In recent years there has been an uptrend on the study of non-local operators, which arise naturally from phenomena such as anomalous diffusion, crystal dislocation or even financial movements. This makes these operators very interesting from both a pure mathematical curiosity point of view and as concrete applications.

In this paper we consider the following system with equations with a non-local diffusion operator:
    \begin{equation} \label{1.1}
    \left\{
    \begin{array}{l}
        \displaystyle u_t(x,t)=\int_\Omega J(x-y)u(y,t)\,dy
        +\int_{\mathbb{R}^N\setminus\Omega} J(x-y)\,dy -u(x,t)-\lambda v^{-p} u^{-\alpha}(x,t), \\
        \displaystyle v_t(x,t)=\int_\Omega J(x-y)v(y,t)\,dy
        +\int_{\mathbb{R}^N\setminus\Omega} J(x-y)\,dy-v(x,t)-\mu u^{-q} v^{-\beta}(x,t),  \\
        u(x,0)=u_0(x)>0; \; v(x,0)=v_0(x)>0,
    \end{array}
    \right.
    \end{equation}
with $x\in\overline{\Omega}$ and $t\in [0,T)$. We consider that $T\in (0,\infty]$ is the maximal existence time of the solution, $\Omega \subset \mathbb{R}^N$ is an open bounded connected smooth domain, the initial data $u_0$ and $v_0$ are positive continuous functions in $\mathcal{C}(\overline{\Omega})$ and the parameters $\lambda,\mu, p, q, \alpha, \beta >0$. The kernel $J:\mathbb{R}^N\to
\mathbb{R}$ is a non-negative $C^1$ function, radially symmetric, decreasing and with $\int_{\mathbb{R}^N} J(s)\,ds=1$.

Notice that this formulation is equivalent to a more straightforward system considering the following extensions to $\mathbb{R}^N$:
\begin{equation*}
\hat{u}(x,t)=\left\{
    \begin{array}{ll}
        u(x,t), & x \in{\overline{\Omega}} \\
        1, & x \in \mathbb{R}^N \backslash \overline{\Omega}.
    \end{array}
    \right., \;\;
    \hat{v}(x,t)=\left\{
    \begin{array}{ll}
        v(x,t), & x \in{\overline{\Omega}} \\
        1, & x \in \mathbb{R}^N \backslash \overline{\Omega}.
    \end{array}
    \right.
\end{equation*}

Then these new functions satisfy the following equations:
\begin{equation} \label{1.2}
    \left\{
    \begin{array}{ll}
        \displaystyle \hat{u}_t(x,t)= J \ast \hat{u} (x,t) - \hat{u} (x,t)-\lambda \hat{v} ^{-p}(x,t), &x\in{\overline{\Omega}}, \, t\in[0,T) \\
        \displaystyle \hat{v}_t(x,t)= J \ast \hat{v} (x,t) - \hat{v}(x,t)-\mu \hat{u}^{-q}(x,t), &x\in{\overline{\Omega}}, \, t\in[0,T)  \\
        \hat{u}(x,t) = \hat{v}(x,t) = 1, & x \in \mathbb{R}^N \backslash \overline{\Omega}, \, t\in[0,T) \\
        \hat{u}(x,0)=u_0(x); \; \hat{v}(x,0)=v_0(x), &x\in{\overline{\Omega}},
    \end{array}
    \right.
\end{equation}

Throughout the paper we will use these two formulations indistinctly, and with some abuse of notation we will write the $\hat{u}$ as $u$ directly, always considering that we are extending the solution $u(\cdot,t) \in \mathcal{C} (\overline{\Omega})$ by $1$ in $\mathbb{R}^N \backslash \overline{\Omega}$. This is a natural trick when working with non-local operators, which force us to consider boundary conditions (in this case, of Dirichlet type) in the whole complement of $\overline{\Omega}$ instead of just at $\partial\Omega$, see \cite{AMRT,BFRW,BV,Fi}. It is important to note that the extended functions are not continuous at $\partial\Omega$ in general, see \cite{Ch, ChChR}.

\begin{defi}
We say that $(u,v) \in \mathcal{C}^\infty ([0,T), \mathcal{C}(\overline{\Omega}) \times \mathcal{C}(\overline{\Omega}))$ is a classical solution of  system \eqref{1.1} if it satisfies the equations in \eqref{1.1} pointwise for every $(x,t)\in\overline\Omega \times [0,T)$.
\end{defi}

The existence and uniqueness of solutions of system \eqref{1.1} has already been studied extensively in previous papers, see \cite{AFJ}, making use of the fact that the reaction terms are locally lipschitz. Moreover, we can also define super and subsolutions of the system as follows.

\begin{defi}  \label{defsupersolucion}
    We say that, given $\tau_0 \in (t_0,t_1)$ and $\tau \in (\tau_0,t_1]$, a function $(\overline{u},\overline{v}) \in \mathcal{C}^1([\tau_0,\tau), \mathcal{C}(\overline{\Omega}) \times \mathcal{C}(\overline\Omega))$ is a supersolution of \eqref{1.1} in $[\tau_0,\tau)$ if $\overline{u}(x,t) >0, \;\; \overline{v}(x,t)>0$ for every $(x,t)\in\overline\Omega \times [\tau_0,\tau)$, and it satisfies
    \begin{equation*}
    \left\{
        \begin{array}{lr}
             \displaystyle \overline{u}_t(x,t) \geq \int_\Omega J(x-y) \overline{u} (y,t) dy + \int_{\mathbb{R}^N\backslash\Omega} J(x-y)dy - \overline{u} (x,t) -\lambda \overline{v}^{-p}\overline{u}^{-\alpha}(x,t), \\
             \displaystyle \overline{v}_t(x,t) \geq \int_\Omega J(x-y) \overline{v} (y,t) dy + \int_{\mathbb{R}^N\backslash\Omega} J(x-y)dy - \overline{v} (x,t) -\mu \overline{u}^{-q}\overline{v}^{-\beta}(x,t), \\
             \overline{u}(x,\tau_0) \geq u_0 (x), \;\; \overline{v}(x,\tau_0) \geq v_0 (x)
        \end{array}
    \right.
    \end{equation*}
    for every $(x,t)\in\overline\Omega \times [\tau_0,\tau)$. Conversely, $(\underline{u},\underline{v}) \in \mathcal{C}^1([\tau_0,\tau), \mathcal{C}(\overline{\Omega}) \times \mathcal{C}(\overline\Omega))$ is a subsolution of \eqref{1.1} in $[\tau_0,\tau)$ if $\underline{u}(x,t)>0, \;\; \underline{v}(x,t)>0$ and they fulfill the reverse inequalities for every $(x,t)\in\overline\Omega \times [\tau_0,\tau)$.
\end{defi}

As proven in \cite{AFJ}, system \eqref{1.1} admits a comparison principle with this definition. This is a tool we will need to use extensively.

Notice that both equations of our system have a singular absorption term. This causes a very interesting phenomenon to appear: there could be some point of $\overline\Omega$ in which one of the components of the solution vanishes. If this happens at a finite time $T$, the corresponding absorption term blows up and the classical solution no longer exists. This is known as quenching, which is going to be the main subject of study of this paper.

We say that a solution of system \eqref{1.1} given by $(u,v)$ presents quenching in finite time $T$ if
$$
\liminf_{t\nearrow T} \min \left\{\min_{\overline\Omega} u(\cdot,t),\min_{\overline\Omega}
v(\cdot,t) \right\} = 0.
$$
We will often refer to a solution that presents quenching as a \textit{quenching solution} or just say that the solution quenches.

The phenomenon of quenching, as with the non-local diffusion, has multiple modeling applications and appears naturally in physical models such as the nonlinear heat conduction in solid hydrogen, see \cite{R}, or the Arrhenius Law in combustion theory, see \cite{CK}.

According to the existence and uniqueness theorem proved in \cite{AFJ}, we can be a little bit more precise in this case. We know that, for system \eqref{1.1}, a solution $(u,v)$ is either global in time or it suffers quenching in finite time $T$. In this case,
\begin{equation}\label{eq-m}
\lim_{t\nearrow T} \min \left\{\min_{\overline\Omega} u(\cdot,t),\min_{\overline\Omega}v(\cdot,t) \right\} = 0,
\end{equation}
that is, the limit inferior is actually a proper limit.

We have three main results in this article. The first of them is related to the existence of both stationary and quenching solutions for system \eqref{1.1}. This is stated in the following theorem.

\begin{Theorem} \label{teo.estacionarias}
There exists an open neighbourhood $U$ of $(0,0)$ in $\mathbb{R}^2$ such that for $(\lambda,\mu)\in ((0,\infty) \times (0,\infty)) \cap \overline{U}$ there exist both global and quenching solutions, whereas for $(\lambda,\mu) \in ((0,\infty) \times (0,\infty)) \cap (\mathbb{R}^2 \backslash \overline{U})$ all solutions present quenching. Moreover, the following are satisfied:

\begin{enumerate} [i)]
    \item if $(\lambda_0,\mu_0) \in ((0,\infty) \times (0,\infty)) \cap \overline{U}$, then $(0,\lambda_0] \times (0,\mu_0] \subset \overline{U}$;
    \item $((0,\infty) \times (0,\infty)) \cap \overline{U} \subset (0,1) \times (0,1)$;
    \item if $(w_1,z_1)$ is a stationary solution of \eqref{1.1} with parameters $(\lambda_1,\mu_1)$, $(w_2,z_2)$ is a stationary solution with parameters $(\lambda_2,\mu_2)$ and $\lambda_1 \leq \lambda_2$, $\mu_1 \leq \mu_2$; then $w_1(x) \geq w_2(x)$ and $z_1(x) \geq z_2(x)$ for every $x\in\overline\Omega$.
\end{enumerate}
\end{Theorem}

Notice that we state some properties for the set $U$, but cannot describe its form or the behaviour of the stationary solutions inside it fully.

The second main result states sufficient conditions for the existence of simultaneous and non-simultaneous quenching. We do not know a priori if both componentes of the solution of \eqref{1.1} suffer quenching at the same time or if there can be conditions under which only one of them suffers quenching while the other remains bounded away from zero. To describe these phenomena, we will make use of the following rigorous definitions of the quenching sets.

\begin{defi}
    Let $(u,v)$ be a quenching solution of \eqref{1.1}. We define its quenching set as
    \begin{equation*}
        Q((u,v)) = \left\{x\in\overline\Omega : \exists \, t_n \rightarrow T^-, x_n \rightarrow x \text{ such that }  \min\{u(x_n,t_n), v(x_n,t_n) \} \rightarrow 0 \right\}.
    \end{equation*}
    Additionally, we can define the quenching sets associated to one of the components:
    \begin{equation*}
        Q(u) = \left\{x\in\overline\Omega : \exists \, t_n \rightarrow T^-, x_n \rightarrow x \text{ such that } u(x_n,t_n) \rightarrow 0 \right\}.
    \end{equation*}
    \begin{equation*}
        Q(v) = \left\{x\in\overline\Omega : \exists \, t_n \rightarrow T^-, x_n \rightarrow x \text{ such that } v(x_n,t_n) \rightarrow 0 \right\}.
    \end{equation*}
\end{defi}

With these sets well defined, we can state our second main theorem on the existence of simultaneous or non-simultaneous quenching.

\begin{Theorem}\label{teo.simultaneo}
Let $(u,v)$ be a quenching solution of \eqref{1.1}. Then,

i) if $p- \beta,q- \alpha\geq 1$, quenching is always simultaneous and $Q(u)=Q(v)$;

ii) if $q-\alpha\ge 1>p-\beta$, there exists $\delta>0$ such that $u(x,t) \geq \delta$ for every $(x,t)\in\overline\Omega \times [0,T)$;

iii) if $p-\beta\ge 1>q - \alpha$, there exists $\delta>0$ such that $v(x,t) \geq \delta$ for every $(x,t)\in\overline\Omega \times [0,T)$;

iv) for $p-\beta,q-\alpha<1$ there can be both simultaneous and non-simultaneous quenching.
\end{Theorem}

We observe that the main parameters that determine if simultaneous or non-simultaneous quenching happens are $p-\beta$ and $q-\alpha$, which are the difference of the exponents of each of the components of the solution in both equations. Whenever the exponents of the coupled components in the singular absorption terms, which are $p$ and $q$, are both sufficiently bigger than the exponents of the uncoupled terms, $\alpha$ and $\beta$, both absorption terms dominate the behaviour of the solution and both components vanish in finite time. When one of these exponents $p,q$ becomes small enough, that absorption term becomes weaker and the corresponding component of the solution can stay bounded away from zero.

Once we know the conditions under which simultaneous and non-simultaneous quenching happen, we are interested in knowing the profile with which the quenching components of the solutions vanish. This is known as the quenching rate of the solutions. To simplify the notation, we say that
\begin{equation*}
    f(t) \sim g(t) \iff \; \exists \, C_1,C_2>0, t_0 \in [0,T) : \, C_1 f(t) \leq g(t) \leq C_2 f(t) \text{ for } t\in [t_0,T),
\end{equation*}
and define
\begin{equation*}
    \min_{x\in\overline\Omega} u(x,t) = u(x_u(t),t), \;\;\;\; \min_{x\in\overline\Omega} v(x,t) = v(x_v(t),t).
\end{equation*}
We state our results in the last two main theorems.

\begin{Theorem}\label{teo.tasas.mosimultena}
Let $(u,v)$ be a  quenching solution of \eqref{1.1}.

i) Assume that  $v(x,t) \geq \delta > 0$ for every $(x,t)\in\overline\Omega \times [0,T)$. Then $q-\alpha<1$ and
$$
\min_{x\in\overline\Omega} u(x,t)\sim (T-t)^{\frac{1}{1+\alpha}}.
$$

ii) Assume that  $u(x,t)\geq \delta >0$ for every $(x,t)\in\overline\Omega \times [0,T)$. Then $p-\beta<1$ and
$$
\min_{x\in\overline\Omega} v(x,t)\sim (T-t)^{\frac{1}{1+\beta}}.
$$
\end{Theorem}

\begin{Theorem} \label{lema.qrate.sim}
Let $(u,v)$ be a quenching solution of \eqref{1.1}. Then

$i)$ for $p-\beta,q-\alpha>1$,
$$
u(x_u (t),t)\sim (T-t)^{\frac{p-1-\beta}{pq-(1+\alpha)(1+\beta)}}, \qquad v(x_u (t),t)\sim (T-t)^{\frac{q-1-\alpha}{pq-(1+\alpha)(1+\beta)}}
$$
and
$$
u(x_v (t),t)\sim (T-t)^{\frac{p-1-\beta}{pq-(1+\alpha)(1+\beta)}}, \qquad v(x_v (t),t)\sim (T-t)^{\frac{q-1-\alpha}{pq-(1+\alpha)(1+\beta)}}.
$$

$ii)$ for $p-\beta>1=q-\alpha$,
$$
u(x_u (t),t)\sim (T-t)^{\frac{1}{1+\alpha}} |\log (T-t)|^{\frac{-p}{(1-p-\beta)(1+\alpha)}}, \qquad v(x_u (t),t)\sim |\log (T-t)|^{\frac{1}{1-p+\beta}}
$$
and
$$
u(x_v (t),t)\sim (T-t)^{\frac{1}{1+\alpha}} |\log (T-t)|^{\frac{-p}{(1-p-\beta)(1+\alpha)}}, \qquad v(x_v (t),t)\sim |\log (T-t)|^{\frac{1}{1-p+\beta}}.
$$

$iii)$ for $p-\beta=1<q-\alpha$,
$$
u(x_u (t),t)\sim |\log (T-t)|^{\frac{1}{1-q+\alpha}}, \qquad v(x_u (t),t)\sim (T-t)^{\frac{1}{1+\beta}} |\log (T-t)|^{\frac{-q}{(1-q-\alpha)(1+\beta)}}.
$$
and
$$
u(x_v (t),t)\sim |\log (T-t)|^{\frac{1}{1-q+\alpha}}, \qquad v(x_v (t),t)\sim (T-t)^{\frac{1}{1+\beta}} |\log (T-t)|^{\frac{-q}{(1-q-\alpha)(1+\beta)}}.
$$

$iv)$ for $p-\beta=1=q-\alpha$,
$$
u(x_u (t),t)\sim (T-t)^{\frac{\lambda}{\lambda+\mu}}, \qquad v(x_u (t),t)\sim (T-t)^{\frac{\mu}{\lambda+\mu}}
$$
and
$$
u(x_v (t),t)\sim (T-t)^{\frac{\lambda}{\lambda+\mu}}, \qquad v(x_v (t),t)\sim (T-t)^{\frac{\mu}{\lambda+\mu}}.
$$

$v)$ for $p-\beta,q-\alpha< 1$, if  $(u,v)$ is a solution of \eqref{1.1} that presents quenching at time $T<+\infty$ such that there exists some $t_0\in[0,T)$ for which $x_u(t) = x_v(t) = x(t)$ for every $t\in[t_0,T)$ and the quenching is simultaneous, then
$$
u(x(t),t)\sim (T-t)^{\frac{p-1-\beta}{pq-(1+\alpha)(1+\beta)}}, \qquad
v(x(t),t)\sim (T-t)^{\frac{q-1-\alpha}{pq-(1+\alpha)(1+\beta)}}.
$$
\end{Theorem}

The paper is organized as follows. In Section 2 we give some preliminary results on the behaviour of the solutions of \eqref{1.1} and prove the existence of both stationary solutions and quenching solutions. Section 3 is devoted to the study of how those stationary solutions behave and to prove necessary and sufficient conditions for them to exist, proving Theorem \ref{teo.estacionarias}. In Section 4, we give conditions under which both types of quenching, simultaneous and non-simultaneous, can appear to prove Theorem \ref{teo.simultaneo}. Finally, in Section 5, we study the quenching rates of the solutions of \eqref{1.1} and prove Theorems \ref{teo.tasas.mosimultena} and \ref{lema.qrate.sim}.

\end{section}

\begin{section}{Some preliminary results}

    In this section we will prove some preliminary results that will be useful in the sequel. First, we state the comparison result for \eqref{1.1} as seen in \cite{AFJ}.

\begin{lema} \label{lemacomparacion}
    Let $(\overline{u},\overline{v})$ be a supersolution of \eqref{1.1} in $[\tau_0,\overline{\tau})$ and $(\underline{u},\underline{v})$ a subsolution of \eqref{1.1} in $[\tau_0,\underline{\tau})$, and take $\tau = \min\{\overline\tau, \underline\tau\}$. Then $\overline{u} (x,t) \geq \underline{u} (x,t)$ and $\overline{v} (x,t) \geq \underline{v} (x,t)$ for every $(x,t)\in\overline\Omega \times [\tau_0,\tau)$.
\end{lema}

    We will use this result throughout the whole paper. Thanks to it, the following corollary gives us an upper bound of every solution of system \eqref{1.1}.

\begin{cor} \label{lema.MNsuper}
    Let $u_0$ and $v_0$  be two positive functions in $C(\overline{\Omega})$ and $(u,v) \in \mathcal{C}^1([0,T), \mathcal{C}(\overline{\Omega})\times \mathcal{C}(\overline{\Omega}))$ the solution of system \eqref{1.1} with initial data $(u_0,v_0)$. Define the quantities
    \begin{equation*}
        M=\max\{1,\|u_0\|_\infty\}, \;\;\; N=\max\{1,\|v_0\|_\infty\}.
    \end{equation*}
    Then $u(x,t) \leq M$ and $v(x,t) \leq N$ for every $(x,t)\in \overline\Omega \times[0,T)$, with $T>0$ the maximal existence time of the solution.
\end{cor}
\begin{proof}
    We can prove that $(M,N)$ is a supersolution of system \eqref{1.1} with initial data $(u_0,v_0)$. On one hand, we see clearly that $M\geq u_0(x)$ and $N \geq v_0(x)$ for every $x\in\overline\Omega$. On the other hand, since $M, N\geq 1$, we see that
    \begin{equation*}
        M \int_{\Omega} J(x-y) dy + \int_{\mathbb{R}^N \backslash \Omega} J(x-y) dy - M - \lambda N^{-p} M^{-\alpha} \leq M - M - \lambda N^{-p} M^{-\alpha} \leq 0.
    \end{equation*}
    The other inequality is analogous. Therefore, thanks to Lemma \ref{lemacomparacion}, we know that
    \begin{equation*}
        u(x,t) \leq M, \;\;\; v(x,t) \leq N
    \end{equation*}
    for all $(x,t)\in\overline{\Omega}\times [0,\tau)$.
\end{proof}

The same comparison result also gives us some a priori estimates of the velocity of the solutions of \eqref{1.1} that we will use later.

\begin{cor} \label{lemamenorque0}
Let $(u,v)$ be a solution of \eqref{1.1} with initial data $(u_0,v_0)$. If the initial data satisfy, for every $x\in\overline\Omega$,
$$
\begin{array}{l}
\displaystyle\int_\Omega J(x-y)u_0(y)\,dy
+\int_{\mathbb{R}^N\setminus\Omega} J(x-y)\,dy -u_0(x)-\lambda v_0^{-p}(x)u_0^{-\alpha}(x)\le 0\\
\displaystyle\int_\Omega J(x-y)v_0(y)\,dy
+\int_{\mathbb{R}^N\setminus\Omega} J(x-y)\,dy -v_0(x)-\mu u_0^{-q}(x)v_0^{-\beta}(x)\le0
\end{array}
$$
then $u_t(x,t)\leq 0$ and $v_t(x,t) \leq 0$ for every $x\in\overline{\Omega}$ and every $t \in[0,T)$, with $T>0$ the maximal existence time of the solution.
\end{cor}

\begin{proof}
Notice that, since $u$ and $v$ are $\mathcal{C}^2$ in the time variable, $(u_t,v_t)$ satisfies the following equations in $\overline\Omega \times [0,\tau)$:
\begin{align*}
    (u_t)_t (x,t) &= \int_\Omega J(x-y) u_t(y,t) dy - u_t(x,t) +\lambda p v^{-p-1} u^{-\alpha} v_t (x,t) + \lambda \alpha v^{-p} u^{-\alpha-1} u_t(x,t), \\
    (v_t)_t (x,t) &= \int_\Omega J(x-y) v_t(y,t) dy - v_t(x,t) +\mu q u^{-q-1} v^{-\beta} u_t (x,t) + \mu \beta u^{-q} v^{-\beta-1} v_t(x,t),
\end{align*}
and the initial data fulfill $u_t(x,0) \leq 0$, $v_t(x,0) \leq 0$ thanks to the hypotheses.

This system of equations also satisfies the comparison result from \cite{AFJ} and it is easy to see that $(0,0)$ is a supersolution of this system. Then we can apply it to get the desired result.
\end{proof}

The second half of this section will be devoted to two very important lemmas that will prove that there exist both quenching solutions and stationary solutions under certain circumstances. First, let us prove that there exists quenching in finite time as long as the initial data are small enough at one point.

\begin{lema} \label{condicionquenching}
    If $u_0,v_0\in \mathcal{C}(\overline{\Omega})$ are two positive functions such that at some point $x_0\in \overline{\Omega}$ they satisfy
    \begin{equation} \label{condquenching}
        u_0(x_0) \leq \left( \frac{\mu}{N^{1+\beta}+\varepsilon} \right)^{1/q}, \;\;\; v_0(x_0) \leq \left( \frac{\lambda}{M^{1+\alpha}+\varepsilon} \right)^{1/p},
    \end{equation}
    where $M=\max\{1,\|u_0\|_\infty\}$, $N=\max\{1,\|v_0\|_\infty\}$ and $\varepsilon>0$. Then the solution of system \eqref{1.1} with initial data $(u_0,v_0)$ quenches in finite time
    \begin{equation*}
        T_\varepsilon\le \min\left\{\frac{u_0^{1+\alpha}(x_0)}{\varepsilon(1+\alpha)},\frac{v_0^{1+\beta}(x_0)}{\varepsilon(1+\beta)}\right\}.
    \end{equation*}
    Furthermore, $u^\alpha u_t(x_0,t) < -\varepsilon$ and $v^\beta v_t(x_0,t) < -\varepsilon$ for every $t\in[0,T_\varepsilon)$.
\end{lema}

\begin{proof}
First we note that thanks to Corollary \ref{lema.MNsuper},
    \begin{equation*}
        u(x,t) \leq M, \;\;\; v(x,t) \leq N
    \end{equation*}
for all $(x,t)\in\overline\Omega\times[0,T_\varepsilon)$, where $T_\varepsilon$ is the maximal definition time of the solution.  On the other hand, by \eqref{condquenching} and the fact that $u_0$ and $v_0$ are strictly positive functions,
    \begin{equation*}
        u^\alpha u_t(x_0,0) < (J\ast u_0(x_0)) u_0^\alpha(x_0) - \lambda v_0(x_0)^{-p} \leq M^{1+\alpha} - \lambda \left(\frac{\lambda}{M^{1+\alpha}+\varepsilon} \right)^{-1} =  -\varepsilon.
    \end{equation*}
    \begin{equation*}
        v^\beta v_t(x_0,0) < (J\ast v_0(x_0)) v_0^\beta(x_0) - \mu u_0(x_0)^{-q} \leq N^{1+\beta} - \mu \left(\frac{\mu}{N^{1+\beta}+\varepsilon} \right)^{-1} = -\varepsilon
    \end{equation*}

    Now suppose that there exists a time $t_1\in(0,T_\varepsilon)$ in which either $u_t(x_0,t)$ or $v_t(x_0,t)$ reach $0$ for the first time. Without loss of generality, $u_t(x_0,t_1)=0$, $u_t(x_0,t)<0$ for all $t \in [0, t_1)$, and $v_t(x_0,t) < 0 $ for all $t\in[0,t_1)$.
    In this situation, it is clear that $v(x_0,t)<v_0(x_0)$ and $u(x_0,t) < u_0(x_0)$ for $t\in[0,t_1)$. Then we conclude that
    \begin{equation*}
       0=u^\alpha u_t(x_0,t_1) < (J\ast u(x_0,t_1))u^\alpha(x_0,t_1) - \lambda v^{-p} (x_0,t_1) < M^{1+\alpha} - \lambda v_0(x_0)^{-p}  \leq -\varepsilon,
    \end{equation*}
    which brings us to a contradiction. Therefore, $u(x_0,\cdot)$ and $v(x_0,\cdot)$ are decreasing functions in $[0,T_\varepsilon)$ and applying the same argument as before,
    $$
    u^\alpha u_t(x_0,t) <-\varepsilon,\qquad v^\beta v_t(x_0,t) <-\varepsilon
    $$
    for every $t\in[0,T_\varepsilon)$. Integrating these inequalities between $0$ and $t\in[0,T_\varepsilon)$, we get
    $$
    u^{1+\alpha}(x_0,t)<u_0^{1+\alpha}(x_0)-\varepsilon (1+\alpha) t,\qquad
    v(x_0,t)<v_0^{1+\beta}(x_0)-\varepsilon(1+\beta) t.
    $$
    Therefore, the solution quenches at finite time
    $$
    T_\varepsilon\le \min\left\{\frac{u_0^{1+\alpha}(x_0)}{\varepsilon(1+\alpha)},\frac{v_0^{1+\beta}(x_0)}{\varepsilon(1+\beta)}\right\}.
    $$
\end{proof}

Observe that there exist quenching solutions for any values of $\lambda,\mu,p,q,\alpha,\beta>0$, since the only condition we need for their existence is for the initial data to be below some threshold that depends on said parameters. This will not be the case for the stationary solutions.

Notice that, if $\lambda=\mu=0$, the solution $(u,v) \equiv (1,1)$ is a stationary solution. Let us prove that if $\lambda$ and $\mu$ are small enough, there also exist stationary solutions.

\begin{lema} \label{lema.existenciaestacionarias}
    There is a neighbourhood of $(0,0)$, $U^\prime \subset \mathbb{R}^2$, such that there exist stationary solutions of system \eqref{1.1}  if $(\lambda,\mu) \in  U^\prime \cap ((0,\infty) \times (0,\infty))$.
\end{lema}

\begin{proof}
    This result follows from the Implicit Function Theorem. First we note that, if $\lambda = \mu = 0$, then $w_0\equiv z_0 \equiv 1$ is a stationary solution. We linearize around this solution, considering:
    \begin{equation*}
        \phi = 1 - u, \;\;\; \gamma = 1 - v.
    \end{equation*}
    Given small $\delta,\varepsilon >0$, let
    \begin{equation*}
    Y_0 = \{ (\phi,\gamma) \in \mathcal{C}(\overline{\Omega}) \times \mathcal{C}(\overline{\Omega}) : -\delta < \phi,\gamma < 1   \}
    \end{equation*}
    and  $\mathcal{F}: (-\varepsilon,\varepsilon) \times (-\varepsilon,\varepsilon) \times Y_0 \longrightarrow \mathcal{C}(\overline{\Omega}) \times \mathcal{C}(\overline{\Omega})$ be a non-linear operator given by
    \begin{align*}
        \mathcal{F}(\lambda,\mu,\phi,\gamma)(x) =& \left(\int_{\Omega} J(x-y)\phi(y) dy - \phi(x) + \lambda (1-\gamma(x))^{-p}(1-\phi(x))^{-\alpha}, \right.\\
        &\left. \int_{\Omega} J(x-y)\gamma(y) dy - \gamma(x) + \mu (1-\phi(x))^{-q} (1-\gamma(x))^{-\beta} \right).
    \end{align*}
    It is clear that $\mathcal{F}(0,0,0,0) = (0,0)$ and the differential operator with respect to the last two variables $D_{(\phi,\gamma)} \mathcal{F}$ evaluated at $\lambda=\mu=0$ is, for every $\xi,\chi\in \mathcal{C}(\overline\Omega)$ and $x\in\overline\Omega$,
    \begin{equation*}
        D_{(\phi,\gamma)} \mathcal{F} (0,0)(\xi,\chi)(x) = \begin{pmatrix}
                                        (K_J - I)(\xi)(x) & 0 \\
                                        0 & (K_J - I)(\chi)(x)
                                     \end{pmatrix},
    \end{equation*}
    where $K_J: \mathcal{C}(\overline\Omega) \longrightarrow \mathcal{C}(\overline\Omega)$ is defined as
    \begin{equation*}
        K_J (z)(x) = \int_\Omega J(x-y) z(y) dy.
    \end{equation*}

    $K_J - I$ is clearly continuous and linear. Let us prove that it is also an injective operator.

    Take $z\in\mathcal{C}(\overline\Omega)$ such that $K_J (z) = z$ and let us prove that $z\equiv 0$. Define $x_0$ as $z(x_0) = \max_{x\in\overline\Omega} z(x)$ and take some $\delta>0$ such that $B(0,\delta) \subset \text{supp}(J)$. Suppose that $z(x_0)>0$. Under these assumptions we know that
    \begin{equation*}
        z(x_0) = K_J(z)(x_0)= \int_\Omega J(x_0-y)z(y)dy \leq \int_\Omega J(x_0-y)z(x_0)dy \leq z(x_0).
    \end{equation*}
    Therefore,
    \begin{equation*}
        \int_\Omega J(x_0-y) (z(y)-z(x_0)) dy = 0.
    \end{equation*}
    Since $J(x_0-y) > 0$ for every $y \in B(x_0,\delta)$, we conclude that $z(y) = z (x_0)$ for every $y\in B(x_0,\delta)$. The domain $\overline\Omega$ is connected and compact, so for any $\overline{x}\in\overline\Omega$, there exists a chain of points $\{x_0,x_1,x_2, \dots, x_n=\overline{x}\}$ such that $x_i \in B(x_{i-1},\delta)$ for every $i=1,\dots,n$. We know that $z(x_1)=z(x_0)$ thanks to our previous argument and, since $\delta$ only depends on $J$, it is clear that $z(x_2)=z(x_1)$ because $x_2 \in B(x_1,\delta)$. Repeating this, we get that $z(x_0) = z(x_i)$ for every $i=1,\dots,n$ and, in particular, $z(\overline{x})=z(x_0)$. Therefore, $z(x)=z(x_0)$ for every $x\in\overline\Omega$. However, if $z$ is a constant function then $K_J(z)=z$ implies that $\int_\Omega J(x-y) dy = 1$ for every $x\in\overline\Omega$, which is a contradiction. Then we conclude that $z(x)\leq 0$ for every $x\in\overline\Omega$. Applying the same argument to the function $-z$, we obtain that $z(x)\geq 0$ for every $x\in\overline\Omega$, which implies that $z\equiv 0$. Therefore, $K_J - I$ is an injective operator.

    Additionally, $K_J$ is compact, see Section 2.1.3 of \cite{SG}, so using the Fredholm alternative we can further assert that $K_J - I$ is bijective and $D_{(\phi,\gamma)} \mathcal{F}(0,0)$ is too. Finally, using the Open Mapping Theorem, we deduce that $D_{(\phi,\gamma)}\mathcal{F}(0,0): \mathcal{C}(\overline{\Omega}) \times \mathcal{C}(\overline{\Omega}) \longrightarrow \mathcal{C}(\overline{\Omega}) \times \mathcal{C}(\overline{\Omega})$ is an isomorphism.

    Therefore, we can apply the Implicit Function Theorem to ensure that there is a neighbourhood $U^\prime \subset \mathbb{R}^2$ of $(0,0)$ such that, if $(\lambda,\mu) \in U^\prime$, then there is a $(\phi_{\lambda,\mu},\gamma_{\lambda,\mu})\in Y_0$ that satisfies $\mathcal{F}(\lambda,\mu,\phi_{\lambda,\mu},\gamma_{\lambda,\mu}) = (0,0)$, see \cite{CR}. Finally, if we define
    \begin{equation*}
    w_{\lambda,\mu}=\left\{
    \begin{array}{ll}
    1-\phi_{\lambda,\mu}, & x\in\overline\Omega,\\
    1, & x\in \mathbb{R}^N\backslash \overline\Omega,
    \end{array}\right.
    \qquad
    z_{\lambda,\mu}=\left\{
    \begin{array}{ll}
    1-\gamma_{\lambda,\mu}, & x\in\overline\Omega,\\
    1, & x\in \mathbb{R}^N\backslash \overline\Omega,
    \end{array}\right.
    \end{equation*}
    these are stationary solutions of our system with parameters $(\lambda,\mu) \in U^\prime \cap ((0,\infty) \times (0,\infty))$.
\end{proof}

Up until now, we have proven that there exist quenching solutions for any set of parameters and that there are some sufficient conditions over the parameters for the existence of stationary solutions. Next, we need to determine what necessary conditions, if any, we need to impose for system \eqref{1.1} to have stationary solutions.

\end{section}

\begin{section}{Proof of Theorem \ref{teo.estacionarias}}

In this section, we will prove Theorem \ref{teo.estacionarias}. To do so, first we need some lemmas that will help us understand how the solutions of the system behave.

This first result gives us an important property of the behaviour of globally defined solutions of system \eqref{1.1}.

\begin{lema} \label{lema.globalesmenorque1}
    Let $(u,v)$ be a globally defined solution of system \eqref{1.1}. Then there exists a time $t_1 \in [0,\infty)$ for which $u(x,t)\leq 1$ and $v(x,t)\leq 1$ for every $x\in\overline\Omega$ and $t\in[t_1,\infty)$. Furthermore, if $(w,z)$ is a stationary solution of \eqref{1.1}, then $\mu^{1/q} < w(x) \leq 1$ and $ \lambda^{1/p} < z(x) \leq 1$ for every $x\in\overline\Omega$.
\end{lema}

\begin{proof}
    Consider $(u,v)$ a globally defined solution with initial data $(u_0,v_0)$. We will split the proof of the first part of the lemma in multiple cases.

    $i)$ If $\|u_0\|_\infty, \|v_0\|_\infty \leq 1$, then $(U,V)$ with $U \equiv V \equiv 1$ is a supersolution of \eqref{1.1} with these initial data. Using Lemma \ref{lemacomparacion}, we get that $1 \geq u(x,t)$ and $1\geq v(x,t)$ for every $(x,t)\in\overline\Omega \times [0,\infty)$.

    $ii)$ If $\|u_0\|_\infty \leq 1 < \|v_0\|_\infty$, consider the pair of functions
    \begin{equation*}
     U (t) = 1, \;\;\; V (t) = \|v_0\|_\infty - Ct,
    \end{equation*}
    with $0<C\leq \mu \|v_0\|_\infty^{-\beta}$, and define $t_1 \in [t_0,\infty)$ as the time in which $V(t_1)=1$. Then we can prove that $(U, V)$ is a supersolution of \eqref{1.1} in $[0,t_1)$. Indeed, we know that $U(0) \geq u_0(x)$ and $V(0) \geq v_0(x)$ for every $x\in\overline\Omega$, and, for every $t\in[0,t_1)$:
    \begin{equation*}
        \int_\Omega J(x-y) dy + \int_{\mathbb{R}^N\backslash\Omega} J(x-y) dy - 1 - \lambda V(t)^{-p} \leq - \lambda \|v_0\|_\infty^{-p} \leq 0 = U_t (t),
    \end{equation*}
    \begin{equation*}
        V(t) \int_\Omega J(x-y) dy + \int_{\mathbb{R}^N\backslash\Omega} J(x-y) dy - V(t) - \mu V(t)^{-\beta} \leq - \mu \|v_0\|_\infty^{-\beta} \leq -C = V_t (t).
    \end{equation*}
    Then by Lemma \ref{lemacomparacion}, $u(x,t) \leq 1$ and $v(x,t) \leq V(t)$ for every $x\in\overline\Omega$ and $t\in[0,t_1)$. By continuity of the solution, $v(x,t_1) \leq V(t_1) = 1$.
    Since there is a time $t_1 \in [0,\infty)$ in which $u(x,t_1) \leq 1$ and $v(x,t_1) \leq 1$ for every $x\in\overline\Omega$, we are again in case $i)$ and we can use the same argument.

    $iii)$ If $\|v_0\|_\infty \leq 1 < \|u_0\|_\infty$, the reasoning is analogous to that of case $ii)$.

    $iv)$ If $\|u_0\|_\infty, \|v_0\|_\infty > 1$, consider the pair of functions
    \begin{equation*}
     U(t) = \|u_0\|_\infty - \lambda \|v_0\|_\infty^{-p} \|u_0\|_\infty^{-\alpha} t, \;\;\; V(t) = \|v_0\|_\infty - \mu \|u_0\|_\infty^{-q} \|v_0\|_\infty^{-\beta} t,
    \end{equation*}
    and define $t_0$ as the first time in which either $U(t_0)=1$ or $V(t_0)=1$. Without loss of generality, assume $U(t_0)=1$. Then we can prove that $(U,V)$ is a supersolution of \eqref{1.1} in $[0,t_0)$. Indeed, we have that $U(0)\geq \|u_0\|_\infty$, $V(0)\geq \|v_0\|_\infty$ and, since $U(t),V(t)\geq 1$ for every $t\in[0,t_0]$:
\begin{equation*}
    U(t) \int_\Omega J(x-y) dy + \int_{\mathbb{R}^N\backslash\Omega} J(x-y) dy - U(t) - \lambda V(t)^{-p} U(t)^{-\alpha} \leq - \lambda \|v_0\|_\infty^{-p} \|u_0\|_\infty^{-\alpha} = U_t (t),
\end{equation*}
\begin{equation*}
    V(t) \int_\Omega J(x-y) dy + \int_{\mathbb{R}^N\backslash\Omega} J(x-y) dy - V(t) - \mu U(t)^{-q} V(t)^{-\beta} \leq - \mu \|u_0\|_\infty^{-q} \|v_0\|_\infty^{-\beta} = V_t (t).
\end{equation*}
Then by Lemma \ref{lemacomparacion} we know that $u(x,t) \leq U(t)$ and $v(x,t) \leq V(t)$ for every $x\in\overline\Omega$ and $t\in[0,t_0)$. By continuity of these functions, it is also true that $u(x,t_0) \leq U(t_0) =  1$ and $v(x,t_0) \leq V(t_0)$. Then we are under the conditions of case $ii)$, we can use the same argument and the first part of the lemma is proven.

Let us prove now the final statement of the lemma. Let $(w,z)$ be a stationary solution of \eqref{1.1}. Suppose first that $\|w\|_\infty >1$ or $\|z\|_\infty>1$. This is a globally defined solution and, following the previous argument, we reach a contradiction. Suppose now that there exists some $x_1 \in \overline\Omega$ such that $w(x_1) \leq \mu^{1/q}$. Since $(w,z)$ is a stationary solution, we have that
\begin{equation*}
    1 > \int_\Omega J(x_1-y)z(y) dy + \int_{\mathbb{R}^N \backslash \Omega} J(x_1-y) dy - z (x_1) = \mu w^{-q}(x_1)z^{-\alpha}(x_1) \geq \mu w^{-q}(x_1) \geq 1.
\end{equation*}
which is a contradiction. If there exists $x_2\in\overline\Omega$ such that $z(x_2) \leq \lambda^{1/p}$, the argument is analogous and we reach a contradiction as well, so the result is proven.
\end{proof}

Next we will prove an important relation between the solution with constant initial data of value $1$ and the stationary solutions of the system.

\begin{lema}\label{lema.1quencheatodo}
Let $(u_1,v_1)$ be the solution of system \eqref{1.1} with initial data $u(x,0) \equiv v(x,0) \equiv 1$. Then, the solution $(u_1,v_1)$ satisfies either:

(i) it presents quenching in finite time and all solutions present quenching in finite time, or

(ii) it is well-defined for all $t\in[0,\infty)$ and it converges uniformly in space from above to a stationary solution $(w,z)$.
\end{lema}
\begin{proof}
    Assume first that $(u_1,v_1)$ presents quenching in finite time $T$ and let us suppose that there exists $(u,v)$ a global solution of \eqref{1.1}. By Lemma \ref{lema.globalesmenorque1}, there exists a time $t_1\in[0,\infty)$ such that $u(x,t) \leq 1$ and $v(x,t) \leq 1$ for every $t\in[t_1,\infty)$. Therefore, $(u_1(x,t),v_1(x,t))$ is a supersolution of system \eqref{1.1} with initial data $(u(x,t_1),v(x,t_1))$. By Lemma \ref{lemacomparacion}, $u_1(x,t) \geq u(x,t+t_1)$ and $v_1(x,t) \geq v(x,t+t_1)$ for every $t\in[0,T)$. Since $(u_1(x,t),v_1(x,t))$ quenches in time $T$, $(u,v)$ will also quench in some time $\widetilde T\leq T+t_1$, which is a contradiction. Therefore, if $(u_1,v_1)$ presents quenching then every solution of \eqref{1.1} also quenches.

    Assume now that $(u_1,v_1)$ is a global solution. Define $\tilde{u} = (u_1)_t$ and $\tilde{v}=(v_1)_t$. Since $\tilde{u}(x,0)\leq0$ and $\tilde{v} (x,0) \leq 0$, we know that $u_1(x,t)$ and $v_1(x,t)$ are both monotonically non-increasing on $t$ for every $x\in\overline\Omega$ thanks to Corollary \ref{lemamenorque0}. Moreover, $u_1(x,t)$ and $v_1(x,t)$ are both bounded from below by zero for every $(x,t)\in\overline\Omega\times[0,\infty)$. This implies that the solution converges point-wise to some $(w,z)$ in $\overline\Omega$.

    We can prove that $w$ and $z$ are strictly positive. Indeed, suppose that there exists some $\overline{x}\in\overline\Omega$ such that $w(\overline{x})=0$ (the argument is the same if $z(\overline{x})=0$). This means that, for every $\varepsilon>0$, there exists a time $t_\varepsilon \in [0,\infty)$ such that $u_1 (\overline{x}, t) < \varepsilon$ for every $t\geq t_\varepsilon$. Taking $\varepsilon< \mu^{1/q}$ and using Corollary \ref{lema.MNsuper}, we get that
    \begin{equation*}
        \tilde{v} (\overline{x},t) \leq 1 - v_1(\overline{x},t) - \mu \varepsilon^{-q} v_1^{-\beta}(\overline{x},t) \leq 1 - \mu \varepsilon^{-q} < 0
    \end{equation*}
    for every $t\geq t_\varepsilon$. Using the fact that $(u_1,v_1)$ is a global solution, we can then take some time in which $u_1(\overline{x},t)$ and $v_1(\overline{x},t)$ are as small as we want. Using Lemma \ref{condicionquenching}, we get that $(u_1,v_1)$ suffers quenching in finite time, which is a contradiction.

    To prove that $(w,z)$ is a stationary solution of our system, first we take limits on every term of \eqref{1.1} and use the Dominated Convergence Theorem to get, for every $x\in\overline\Omega$:
    \begin{align*}
        \lim_{t\rightarrow\infty} \tilde{u}(x,t) &= \int_\Omega J(x-y)w(y) dy + \int_{\mathbb{R}^N \backslash \Omega} J(x-y) dy - w(x) - \lambda z^{-p} w^{-\alpha}(x). \\
        \lim_{t\rightarrow\infty} \tilde{v}(x,t) &= \int_\Omega J(x-y)z(y) dy + \int_{\mathbb{R}^N \backslash \Omega} J(x-y) dy - z(x) - \mu w^{-q} z^{-\beta}(x).
    \end{align*}
    We know that $u_1$ and $v_1$ are bounded continuous functions and these limits exist, so they both must be zero.

    Then we have to prove that $w,z$ are continuous in $\overline\Omega$. Since $u_1$ and $v_1$ are non-increasing in time, for some fixed $x_0\in \overline\Omega$ and $t_0\in[0,\infty)$, we know that
    \begin{align*}
        \tilde{u}_t (x_0,t) &\leq -\tilde{u}(x_0,t)+p \lambda v_1^{-p-1}(x_0,t_0)u_1^{-\alpha}(x_0,t_0)\tilde{v}(x_0,t) + \alpha\lambda v_1^{-p}(x_0,t_0) u_1^{-\alpha-1}(x_0,t_0) \tilde{u}(x_0,t),\\
        \tilde{v}_t (x_0,t) &\leq -\tilde{v}(x_0,t)+q\mu u_1^{-q-1} (x_0,t_0)v_1^{-\beta}(x_0,t_0) \tilde{u}(x_0,t) + \beta\mu u_1^{-q} (x_0,t_0)v_1^{-\beta-1}(x_0,t_0) \tilde{v}(x_0,t),
    \end{align*}
    for every $t\in[t_0,\infty)$. This means that $(\tilde{u},\tilde{v})(x_0,t)$ is a subsolution of a linear system $\vec{\phi}_t = A \vec{\phi}$ whose constant matrix is
    \begin{equation*}
        A = \begin{pmatrix}
            -1 + \alpha\lambda v_1^{-p}(x_0,t_0) u_1^{-\alpha-1}(x_0,t_0) & p \lambda v_1^{-p-1}(x_0,t_0)u_1^{-\alpha}(x_0,t_0) \\
            q\mu u_1^{-q-1} (x_0,t_0)v_1^{-\beta}(x_0,t_0) & -1 + \beta\mu u_1^{-q} (x_0,t_0)v_1^{-\beta-1}(x_0,t_0)
        \end{pmatrix}.
    \end{equation*}
    Furthermore, we know that $\tilde{u}(x_0,t),\tilde{v}(x_0,t) \leq 0$ and $\tilde{u}(x_0,t),\tilde{v}(x_0,t) \xrightarrow{t} 0$, so the equilibrium point $(0,0)$ of this system is stable, and therefore the eigenvalues of $A$ must be both negative. Then its determinant is positive, which means that
    \begin{align} \label{condicion.positividad.continuidad}
        &(1 - \alpha\lambda v_1^{-p}(x_0,t_0) u_1^{-\alpha-1}(x_0,t_0))(1 - \beta\mu u_1^{-q} (x_0,t_0)v_1^{-\beta-1}(x_0,t_0))\\
        & - pq\lambda\mu v_1^{-p-\beta-1}(x_0,t_0) u_1^{-q-\alpha-1}(x_0,t_0) > 0. \nonumber
    \end{align}
    Note that this is true for every $x_0\in\overline\Omega$ and $t_0\in[0,\infty)$. We can go even further and note that, since $\tilde{u},\tilde{v}\leq 0$, $\tilde{v}$ also verifies that, for fixed $x_0\in \overline\Omega$ and $t_0\in[0,\infty)$,
    \begin{equation*}
        \tilde{v}_t(x_0,t) \leq (-1 + \mu \beta u_1^{-q}(x_0,t_0)v_1^{-\beta-1}(x_0,t_0)) \tilde{v}(x_0,t),
    \end{equation*}
    for every $t\in[t_0,\infty)$. Then $\tilde{v}(x_0,t)$ is a subsolution of this other linear equation and, with a similar argument, we get that
    \begin{equation*}
        1 - \mu \beta u_1^{-q}(x_0,t_0)v_1^{-\beta-1}(x_0,t_0) > 0
    \end{equation*}
    for every $(x_0,t_0) \in \overline\Omega \times [0,\infty)$. Note that this together with \eqref{condicion.positividad.continuidad} implies that we also have
    \begin{equation} \label{condicion.positividad.reducida}
        1 - \mu \beta w^{-q}(x_0)z^{-\beta-1}(x_0) > 0
    \end{equation}
    for every $x_0\in\overline\Omega$. Indeed, if this wasn't the case, this quantity would vanish and, taking limits in time in \eqref{condicion.positividad.continuidad}, we would have that
    \begin{equation*}
        - pq\lambda\mu z^{-p-\beta-1}(x_0) w^{-q-\alpha-1}(x_0) \geq 0,
    \end{equation*}
    which is a contradiction.

    Next, we recall that we have already proved that
     \begin{align*}
        \int_\Omega J(x-y)w(y) dy + \int_{\mathbb{R}^N \backslash \Omega} J(x-y) dy &= w(x) + \lambda z^{-p} w^{-\alpha}(x). \\
        \int_\Omega J(x-y)z(y) dy + \int_{\mathbb{R}^N \backslash \Omega} J(x-y) dy &= z(x) + \mu w^{-q} z^{-\beta}(x).
    \end{align*}
    Since the left hand side of these equations is continuous in $\overline\Omega$, the right hand side is continuous as well. Suppose without loss of generality that $w(x)$ is not continuous at some point $\hat{x} \in \overline\Omega$, that is, there exists some $\{x_n\}_n \subset \overline\Omega$ such that
    \begin{equation*}
        x_n \longrightarrow \hat{x}, \;\; w(x_n) \longrightarrow A > w(\hat{x}), \;\; z(x_n) \longrightarrow B,
    \end{equation*}
    as $n\to \infty$ with some $A,B>0$ (if $A<w(\hat{x})$ the argument is very similar). Using that $(w+\lambda z^{-p}w^{-\alpha})(x)$ is continuous and applying the Mean Value Theorem, we get
    \begin{align} \label{inequality.assuming.discontinuity.1}
        0 =& \, (A-w(\hat{x})) + \lambda(A^{-\alpha}B^{-p} - w^{-\alpha}(\hat{x})z^{-p}(\hat{x})) \\
        >& \,(A- w(\hat{x})) - \lambda \alpha w^{-\alpha-1}(\hat{x}) z^{-p}(\hat{x})(A-w(\hat{x})) - \lambda p w^{-\alpha}(\hat{x})z^{-p-1}(\hat{x})(B-z(\hat{x})) \nonumber
    \end{align}
    Then doing the same with the continuous function $(z+\mu w^{-q}z^{-\beta})(x)$:
    \begin{align} \label{inequality.assuming.discontinuity.2}
        0 =& \, (B-z(\hat{x})) + \mu(A^{-q}B^{-\beta} - w^{-q}(\hat{x})z^{-\beta}(\hat{x})) \\
        >& \,(1 -\mu \beta w^{-q}(\hat{x})z^{-\beta-1}(\hat{x}))(B- z(\hat{x})) - \mu q w^{-q-1}(\hat{x}) z^{-\beta}(\hat{x})(A-w(\hat{x})). \nonumber
    \end{align}
    Plugging \eqref{inequality.assuming.discontinuity.2} into \eqref{inequality.assuming.discontinuity.1}, recalling the strict positivity of the quantity \eqref{condicion.positividad.reducida}, we get:
    \begin{align*}
        0 >& \, (1 - \lambda \alpha w^{-\alpha-1}(\hat{x}) z^{-p}(\hat{x}))(1 -\mu \beta w^{-q}(\hat{x})z^{-\beta-1}(\hat{x})) (A-w(\hat{x})) \\
        &\, - \lambda\mu pq w^{-q-\alpha-1}(\hat{x})z^{-p-\beta-1}(\hat{x})(A-w(\hat{x})).
    \end{align*}
    This implies that
    \begin{equation*}
        (1 - \lambda \alpha w^{-\alpha-1}(\hat{x}) z^{-p}(\hat{x}))(1 -\mu \beta w^{-q}(\hat{x})z^{-\beta-1}(\hat{x}))- \lambda\mu pq w^{-q-\alpha-1}(\hat{x})z^{-p-\beta-1}(\hat{x}) := -\delta < 0.
    \end{equation*}
    Since $w$ and $z$ are the limits in time of $u_1$ and $v_1$, there exists $t_1\in[0,\infty)$ such that
    \begin{equation*}
        (1 - \lambda \alpha u_1^{-\alpha-1} v_1^{-p}(\hat{x},t_1))(1 -\mu \beta u_1^{-q}v_1^{-\beta-1}(\hat{x},t_1)) - \lambda\mu pq u_1^{-q-\alpha-1}v_1^{-p-\beta-1}(\hat{x},t_1) \leq -\frac{\delta}{2} < 0,
    \end{equation*}
    which is a contradiction with inequality \eqref{condicion.positividad.continuidad}.

    We have concluded that $(w,z)$ is a continuous stationary solution of system \eqref{1.1}. Finally, since $(u_1,v_1)(x,t)$ are continuous, non-increasing in $t$ and they converge pointwise to $(w,z)(x)$, which are also continuous, Dini's Theorem gives us that the convergence is actually uniform from above.
\end{proof}

We can use these lemmas to prove Theorem \ref{teo.estacionarias}.

\begin{proof}[Proof of Theorem \ref{teo.estacionarias}]
    Lemma \ref{lema.existenciaestacionarias} gives us a small neighbourhood around the origin $U^\prime \subset \mathbb{R}^2$ in which the system accepts stationary solutions. Moreover, it follows from Lemma \ref{lema.globalesmenorque1} that there can be no stationary solutions in our system if $\lambda \geq 1$ or $\mu \geq 1$.

    Next we notice that, if there exists $(w_0,z_0)$ a stationary solution of \eqref{1.1} with parameters $(\lambda_0,\mu_0)$, then it is a subsolution of \eqref{1.1} for parameters $(\lambda,\mu)$ with $\lambda\leq \lambda_0$ and $\mu\leq \mu_0$. Then define $(u_1,v_1)$ as the solution of \eqref{1.1} with parameters $(\lambda,\mu)$ and initial data $u(x,0)\equiv v(x,0) \equiv 1$. Then, due to Lemma \ref{lemacomparacion}, $(u_1,v_1)$ is bounded from below by $(w_0,z_0)$. Using Lemma \ref{lema.1quencheatodo}, $(u_1,v_1)$ converges to a stationary solution $(w,z)$. Moreover $w(x)\ge w_0(x)$, $z(x)\ge z_0(x)$, which gives us a monotonicity property of these stationary solutions with respect to the parameters $(\lambda,\mu)$.

    Consider now a fixed $\lambda$ and define
    \begin{equation*}
    \mu_\lambda^* = \sup\{ \mu : \text{a stationary solution exists for } (\lambda,\mu) \}.
    \end{equation*}
    Thanks to Lemma \ref{lema.globalesmenorque1} and the monotonicity of stationary solutions with respect to $(\lambda,\mu)$, it is clear that, for every $x\in\overline\Omega$:
    \begin{equation*}
        w_{\lambda,\mu_\lambda^*} (x) = \lim_{\mu\rightarrow \mu_\lambda^*} w_{\lambda,\mu} (x) > 0, \;\;\; z_{\lambda,\mu_\lambda^*} (x) = \lim_{\mu\rightarrow \mu_\lambda^*} z_{\lambda,\mu} (x) > 0.
    \end{equation*}
    Then taking the limit $\mu \to \mu_\lambda^\ast$ in the equations of \eqref{1.1}, we see that $(w_{\lambda,\mu_\lambda^*}, z_{\lambda,\mu_\lambda^*} )$ is a stationary solution of \eqref{1.1} with parameters $(\lambda, \mu^\ast_\lambda)$.

    We conclude that there exists a neighbourhood $U$ of $(0,0)$ in $\mathbb{R}^2$ such that there are stationary solutions if and only if $(\lambda, \mu) \in \overline{U} \cap ((0,\infty) \times (0,\infty))$. Finally, since there are no stationary solutions for $(\lambda,\mu) \notin \overline{U} \cap ((0,\infty) \times (0,\infty))$, every solution there quenches by Lemma \ref{lema.1quencheatodo}.
\end{proof}

\end{section}

\begin{section}{Simultaneous and Non-simultaneous quenching}

To prove Theorem~\ref{teo.simultaneo} we need to consider two different cases: $\max\{p-\beta,q-\alpha\}\ge 1$ and $\max\{p-\beta,q-\alpha\}<1$. We will be able to prove that a functional inequality will give us all the information in the former case, while the latter is much more involved. This happens because, whenever $\max\{p-\beta,q-\alpha\}<1$, the coupled absorption terms are both sufficiently weak so that, choosing appropriate initial data, both simultaneous and non-simultaneous quenching coexist.

\subsection{The case $\max\{p-\beta,q-\alpha\}\geq1$.}

We will start this case by proving a key inequality.

\begin{prop} \label{prop.des.clave}
    Let $(u,v)$ be a solution of system \eqref{1.1} defined up to time $T\in(0,+\infty]$ and consider the function $\Psi_a: \mathcal{C}^1([0,T),\mathcal{C}(\overline\Omega)) \longrightarrow \mathcal{C}^1([0,T),\mathcal{C}(\overline\Omega))$ defined as
    \begin{equation}\label{primitiva}
        \Psi_a[g](x,t)=\left\{
        \begin{array}{ll}
        \frac{g^{1-a}}{1-a}(x,t) & a\ne1\\
        \ln |g(x,t)|\qquad & a=1.
        \end{array}\right.
    \end{equation}
    Then there exists some $D>0$ such that, for every $(x,t)\in\overline\Omega\times[0,T)$:
    \begin{equation}\label{des.clave}
        \left|\mu \Psi_{q-\alpha}[u](x,t)-\lambda\Psi_{p-\beta}[v](x,t)\right|\le D.
    \end{equation}
\end{prop}

\begin{proof}
    Consider $(u,v)$ a solution of \eqref{1.1} defined up to time $T\in(0,+\infty]$ with initial data $(u_0,v_0)$ and note that
    $$
    \mu u^{\alpha-q} u_t (x,t)- \lambda v^{\beta-p} v_t (x,t)= \mu u^{\alpha-q} \left( J*u-u \right)(x,t)- \lambda v^{\beta-p} \left( J*v-v\right)(x,t),
    $$
    for every $(x,t)\in\overline\Omega \times [0,T)$. Moreover, we know that  $0< \min\{u_0(x),v_0(x)\}\le C$ and Corollary \ref{lema.MNsuper} implies that $u(x,t)\le M$ and $v(x,t)\le N$ for every $(x,t) \in \overline\Omega \times [0,T)$. Then
$$
\begin{array}{rl}
\displaystyle\left|\int_0^t u^{\alpha-q} \left( J*u-u \right)(x,s)ds\right|\le &
\displaystyle M^{1+\alpha} \int_0^t u^{-q} (x,s)ds
= \frac{M^{1+\alpha}}{\mu} \int_0^t v^\beta (J*v-v-v_t)(x,s)ds\\
\le &\displaystyle\frac{M^{1+\alpha}N^{1+\beta}}{\mu} \left(T+\frac{1}{1+\beta}\right),
\end{array}
$$
which means that there exists a constant $D_1>0$ such that
\begin{equation*}
    \left|\int_0^t \Big(\mu u^{\alpha-q} u_t (x,s)- \lambda v^{\beta-p} v_t (x,s)\Big) ds\right| \leq D_1.
\end{equation*}
Moreover, we can also get a lower bound for this quantity:
$$
\begin{array}{rl}
\displaystyle\left|\int_0^t \Big(\mu u^{\alpha-q} u_t (x,s)- \lambda v^{\beta-p} v_t (x,s)\Big) ds\right| \ge &
\left|\mu \Psi_{q-\alpha}[u](x,t)-\lambda\Psi_{p-\beta}[v](x,t)\right| \\
& - \left|\mu\Psi_{q-\alpha}[u](x,0)-\lambda\Psi_{p-\beta}[v](x,0)\right|\\
\ge & \left|\mu\Psi_{q-\alpha}[u](x,t)-\lambda\Psi_{p-\beta}[v](x,t)\right|- D_2,
\end{array}
$$
for every $(x,t)\in\overline\Omega \times [0,T)$, where $D_2>0$. Using both bounds, we get our key inequality:
\begin{equation*}
\left|\mu \Psi_{q-\alpha}[u](x,t)-\lambda\Psi_{p-\beta}[v](x,t)\right|\le D,
\end{equation*}
for every $(x,t)\in\overline\Omega \times [0,T)$ and $D = D_1 + D_2 >0$.
\end{proof}

\begin{rem}\label{proposicion vacia}
    Notice that, for $\max\{p-\beta,q-\alpha\}<1$, the estimate in Proposition \ref{prop.des.clave} is empty because in this case the primitives $\Psi_{q-\alpha}[u]$ and $\Psi_{p-\beta}[v]$ are bounded functions. However, for the case $\max\{p-\beta,q-\alpha\}\ge1$, at least one of the primitives may be unbounded and we obtain all the information we need from the inequality.
\end{rem}

\begin{lema}\label{max(p,q)>=1}
Let $\max\{p-\beta,q-\alpha\}\ge 1$ and $(u,v)$ be a quenching solution of system \eqref{1.1} defined up to time $T$. Then,
\begin{enumerate}
\item If $p-\beta \ge 1 > q - \alpha$, then only the component $u$ quenches.
\item If $q-\alpha\ge 1>p-\beta$, then only the component  $v$ quenches.
\item If $p-\beta, q-\alpha\ge 1$, quenching is always simultaneous and $Q(u)=Q(v)$.
\end{enumerate}
\end{lema}
\begin{proof}
Assume first that  $p-\beta\ge1>q-\alpha$ and suppose that $x_0\in Q(v)$. This means that there exists a sequence $\{(x_n,t_n)\}_n$ such that $(x_n,t_n)\to (x_0,T)$ and $v(x_n,t_n)\to 0$ as $n\to\infty$. Notice that $\Psi_{q-\alpha}[u](x,t) \leq C$ for every $(x,t)\in\overline\Omega$, but $\Psi_{p-\beta}[v](x_n,t_n)\to -\infty$ as $n\to\infty$. This is a contradiction with inequality \eqref{des.clave}. We conclude that $v$ cannot present quenching. The case $q-\alpha \geq 1 > p-\beta$ is analogous.

Now assume $p-\beta,q-\alpha\ge 1$ and take some $x_0\in Q(u,v)$. Then there exists a  sequence $\{(x_n,t_n)\}_n$ such that $(x_n,t_n)\to (x_0,T)$ and one of the components tends to zero as $n\to\infty$. Without loss of generality, $u(x_n,t_n)\to 0$ as $n\to\infty$. As $\Psi_{q-\alpha}[u](x_n,t_n)$  is unbounded, inequality \eqref{des.clave} implies that $\Psi_{p-\beta}[v](x_n,t_n)$ is also unbounded.
Then, $v(x_n,t_n)\to 0$, the quenching is always simultaneous and $Q(u)=Q(v)$.
\end{proof}

\begin{cor}
    Let $(u,v)$ be a quenching solution of system $\eqref{1.1}$ that suffers simultaneous quenching. Then either $p-\beta,q-\alpha\geq 1$ or $p-\beta,q-\alpha<1$.
\end{cor}

Observe that Proposition \ref{prop.des.clave} also gives us the following results as a direct consequence.
\begin{cor} \label{lema.auxiliar.nosim}
Let $(u,v)$ be a quenching solution of system \eqref{1.1} and $x_0\in Q(u,v)$. Then,
\begin{enumerate}
\item If only the component $u$ quenches at $x_0$, then $q-\alpha<1$.
\item If only the component $v$ quenches at $x_0$, then $p-\beta<1$.
\end{enumerate}
\end{cor}
\begin{cor}\label{lema-estimaciones}
Let $p-\beta,q-\alpha\ge1$ and $\{(x_n,t_n)\}_n$ a sequence such that both components present quenching as $n\to\infty$. Then,
\begin{enumerate}
\item for $p-\beta,q-\alpha>1$, $u^{1-q+\alpha}(x_n,t_n)\sim v^{1-p+\beta}(x_n,t_n)$;
\item for $p-\beta>1=q-\alpha$, $-\log(u(x_n,t_n))\sim v^{1-p+\beta}(x_n,t_n)$;
\item for $q-\alpha>1=p-\beta$, $u^{1-q+\alpha}(x_n,t_n)\sim -\log(v(x_n,t_n))$;
\item for $p-\beta=q-\alpha=1$, $u^{\mu}(x_n,t_n)\sim v^\lambda(x_n,t_n)$.
\end{enumerate}
\end{cor}

This last lemma gives us some relations between the components of the solutions of \eqref{1.1} that will be very useful to get the quenching rates in Section 5.

\subsection{The case $\max\{p-\beta,q-\alpha\}<1$.}

This case is more involved because non-simultaneous and simultaneous quenching coexist. To prove that we use a shooting argument. Let $(u_0,v_0)$ be two functions such that
\begin{equation} \label{condicion.quenching.iv2.bis.1}
        \|u_0\|_\infty \leq \min \left\{ 1, \left( \frac{\mu}{2} \right)^{1/q}\right\}, \;\;\;\; \|v_0\|_\infty \leq \min \left\{ 1, \left( \frac{\lambda}{2} \right)^{1/p},\right\}.
    \end{equation}
and let $(u_\delta,v_\delta)$  be the solution of system \eqref{1.1} with initial data $(\delta u_0,(1-\delta)v_0)$ for $\delta\in(0,1)$. Notice that the hypotheses of Lemma \ref{condicionquenching} are satisfied for these initial data for all $x\in\overline\Omega$ with $M=N=\varepsilon=1$. Then the solution quenches at finite time
\begin{equation}\label{Tdelta}
T_\delta \leq \min \left\{ \min_{x\in\overline\Omega} \frac{\delta^{1+\alpha}}{1+\alpha}  u_0^{1+\alpha}(x), \min_{x\in\overline\Omega} \frac{(1-\delta)^{1+\beta}}{1+\beta} v_0^{1+\beta}(x) \right\} \leq 1,
\end{equation}
and
\begin{equation}\label{eq.decrece}
u^\alpha_\delta(u_\delta)_t(x,t)\le -1,\qquad  v^\beta_\delta(v_\delta)_t(x,t)\le -1,
\end{equation}
for every $(x,t)\in\overline\Omega \times [0,T_\delta)$. Integrating these inequalities between $t\in[0,T_\delta)$ and $T_\delta$,
\begin{equation}\label{eq.tasainf}
u_\delta (x,t)\ge (T_\delta-t)^{\frac{1}{1+\alpha}}, \qquad
v_\delta (x,t)\ge (T_\delta-t)^{\frac{1}{1+\beta}},
\end{equation}
for every $(x,t)\in\overline\Omega \times [0,T_\delta)$. Now, let us introduce the following disjoint sets:
\begin{equation*}
    \begin{array}{l}
        \displaystyle A^+=\{\delta\in(0,1)\,:\, \mbox{only the component $v_\delta$ quenches} \},\\
        \displaystyle A^-=\{\delta\in(0,1)\,:\, \mbox{only the component $u_\delta$ quenches}\},\\
        \displaystyle A=\{\delta\in(0,1)\,:\, \mbox{both components quench}\}.
    \end{array}
\end{equation*}

\begin{lema} \label{lema.sets.nonempty}
Let $p-\beta,q-\alpha<1$. The sets $A^+$ and $A^-$ are nonempty.
\end{lema}
\begin{proof}
Notice that  by \eqref{eq.tasainf}
$$
\begin{array}{rl}
u_\delta^\alpha(u_\delta)_t (x,t)= &\displaystyle u_\delta^\alpha(x,t)\left(\int_\Omega J(x-y) u_\delta (y,t) + \int_{\mathbb{R}^N \backslash\Omega} J(x-y) dy \right) -u_\delta^{1+\alpha} (x,t)-\lambda v_\delta^{-p} (x,t)\\
\ge&\displaystyle   -u_\delta^{1+\alpha} (x,t) -\lambda (T_\delta -t)^{\frac{-p}{1+\beta}},
\end{array}
$$
for every $(x,t) \in \overline\Omega \times [0,T_\delta)$. Solving the differential inequality, we get
\begin{align*}
e^{(1+\alpha)t} u_\delta^{1+\alpha}(x,t)\ge& \;  u_\delta^{1+\alpha}(x,0)- \lambda(1+\alpha)\int_0^t e^{(1+\alpha)s}(T_\delta-s)^{\frac{-p}{1+\beta}}ds \\
\ge& \; \delta^{1+\alpha} u^{1+\alpha}_0(x)-\frac{\lambda(1+\alpha)(1+\beta)}{1+\beta-p} e^{(1+\alpha)T_\delta} ( T_\delta^{\frac{1+\beta-p}{1+\beta}}-(T_\delta-t)^{\frac{1+\beta-p}{1+\beta}}),
\end{align*}
for every $(x,t) \in \overline\Omega \times [0,T_\delta)$. Therefore,
$$
\lim_{t\nearrow T_\delta} e^{(1+\alpha)t} u_\delta^{1+\alpha}(x,t)\ge \delta^{1+\alpha} u^{1+\alpha}_0(x)- \frac{\lambda(1+\alpha)(1+\beta)}{1+\beta-p} e^{(1+\alpha)T_\delta} T_\delta^{\frac{1+\beta-p}{1+\beta}},
$$
for every $x\in\overline\Omega$. Since $T_\delta\to 0$ as $\delta\to 1$ and $p-\beta<1$, we can take $\delta$ sufficiently close to $1$ to get that
$$
\lim_{t\nearrow T_\delta} e^{(1+\alpha)t} u_\delta^{1+\alpha}(x,t) \ge \frac{\delta^{1+\alpha}}{2} u^{1+\alpha}_0(x) >0,
$$
for every $x\in\overline\Omega$. Therefore only the component $v_\delta$ quenches and the set $A^+$ is nonempty.

With the same argument, it is easy to see that taking $\delta$ sufficiently close to $0$ only the component $u$ quenches and the set $A^-$ is nonempty.
\end{proof}

We also want to prove that $A^+$ and $A^-$ are open sets. To do so, we first need to see that the quenching time is continuous with respect to $\delta$.

\begin{lema}
    Let $p-\beta,q-\alpha<1$ and consider $(u_\delta,v_\delta)$ the solution of system \eqref{1.1} with initial data $(\delta u_0,(1-\delta) v_0)$.  Then the quenching time $T_\delta$ of this solution is continuous with respect to $\delta$.
\end{lema}

\begin{proof}
We take $\delta$ and $\widetilde\delta$ such that $|\delta-\widetilde\delta|\le\mu$ and define the function
\begin{equation*}
m(t) = \min \left\{ \min_{x\in\overline\Omega} u_{\delta}(x,t),  \min_{x\in\overline\Omega} v_{\delta}(x,t),
\min_{x\in\overline\Omega} u_{\widetilde\delta}(x,t),  \min_{x\in\overline\Omega} v_{\widetilde\delta}(x,t) \right\}
\end{equation*}
for $t\in (0,T_0)$ with $T_0=\min\{T_\delta,T_{\widetilde\delta}\}$.
Observe that $m(t)\to 0$ as $t\to T_0$ and from \eqref{eq.decrece} it is a decreasing function. So there exists a time $t_0$ such that $m(t_0)=\varepsilon/3$. Moreover assuming that  $m(t_0)=\min_{x\in\overline\Omega} u_{\delta}(x,t_0)=u_{\delta}(x_0,t_0)$, we use the continuity of the solutions with respect to the initial data in the system \ref{1.1} to obtain that $u_{\widetilde\delta}(x_0,t_0)\le 2\varepsilon/3$ provided that $\mu$ is small enough.
Finally, by \eqref{eq.tasainf} we have
$$
|T_\delta-T_{\widetilde\delta}|\le |T_\delta-t_0|+|T_{\widetilde\delta}-t_0|\le u_{\delta}(x_0,t_0)^{1+\alpha}+u_{\widetilde\delta}(x_0,t_0)^{1+\alpha}\le C\varepsilon^{1+\alpha}
$$
and the result follows.
\end{proof}

\begin{lema} \label{lema.sets.open}
    Let $p-\beta,q-\alpha<1$. Then $A^+$  and $A^-$ are open sets.
\end{lema}
\begin{proof}
We only study $A^+$ and the study of $A^-$ is similar. Notice that for $\delta_0 \in A^+$  there exists $c>0$ such that $u_{\delta_0}(x,t)\ge c$ for every $(x,t) \in \overline\Omega \times [0,T_{\delta_0})$.

Now for a fixed $\varepsilon>0$, we consider $\hat{t} \in (T_{\delta_0} - \varepsilon/2,T_{\delta_0}-\varepsilon/4)$.
By the continuity of the quenching time with respect to $\delta$ there exists some $\mu_1>0$ such that $|T_\delta - T_{\delta_0}| < \varepsilon/4$ provided $|\delta-\delta_0| < \mu_1$. Notice that $\hat t<T_\delta$, in fact
$$
T_\delta-\hat t\le|T_\delta-T_{\delta_0}|+|T_{\delta_0}-\hat t|\le \varepsilon.
$$
Since $\hat t<T_\delta$, the function $u_\delta(x,\hat t)$ is well defined. Then due to
the continuous dependence with respect to the initial data, there exists $\mu_1 \geq \mu_2>0$ such that $u_\delta(x,\hat{t}) \geq c/2$ for every $x\in\overline\Omega$ if $|\delta-\delta_0| < \mu_2$.

Then consider $\delta \in (\delta_0 - \mu_2, \delta_0 + \mu_2)$. Following the same argument as in Lemma \ref{lema.sets.nonempty}, integrating between $\hat{t}$ and $T_\delta$ this time, we get
\begin{align*}
\displaystyle \lim_{t\nearrow {T}_\delta} e^{(1+\alpha)t} u^{1+\alpha}_\delta(x,t)
&\displaystyle
\ge e^{(1+\alpha)\hat{t}} u^{1+\alpha}_\delta(x,\hat{t})- \frac{\lambda(1+\alpha)(1+\beta)}{1+\beta-p} e^{(1+\alpha)T_\delta}( T_\delta-\hat{t})^{\frac{1+\beta-p}{1+\beta}} \\
& \ge e^{(1+\alpha)T_\delta}\left(\left(\frac{ce^{-\varepsilon}}{2}\right)^{1+\alpha} - \frac{\lambda(1+\alpha)(1+\beta)}{1+\beta-p} \varepsilon^{\frac{1+\beta-p}{1+\beta}}\right).
\end{align*}

Then taking $\varepsilon$ small enough we deduce that $u_\delta$ does not quench, that is, $(\delta_0-\mu_2,\delta_0+\mu_2)\subset A^+$ and the result follows.
\end{proof}

Since $A^+$ and $A^-$ are nonempty open sets and $(0,1)$ is a connected interval, it follows that $A$ is a closed nonempty subset of $(0,1)$. Therefore, there exists some initial datum taking $\delta \in A$ such that the solution presents simultaneous quenching, and the solution will present non-simultaneous quenching if we consider initial data taking $\delta\in A^+ \cup A^-$.

\begin{proof}[Proof of Theorem \ref{teo.simultaneo}]
    This follows from Lemma \ref{max(p,q)>=1}, Lemma \ref{lema.sets.nonempty} and Lemma \ref{lema.sets.open}.
\end{proof}

\end{section}

\begin{section}{Quenching Rates}
Our final main objective in this paper is to obtain the quenching rates of the solutions of \eqref{1.1}. To study these quenching rates, we will need to work with the quantities
$$
\min_{x \in \overline\Omega} u(x,t) = u(x_u(t),t), \qquad \min_{x \in \overline\Omega} v(x,t) = v(x_v(t),t).
$$
These functions of $t$ are differentiable for almost every time and it can be shown that $\partial_t (u(x_u(t),t)) = u_t (x_u(t),t)$ and $\partial_t (v(x_v(t),t)) = v_t (x_v(t),t)$. Note that this allows us to integrate the functions $u_t(x_u(t),t)$ and $v_t(x_v(t),t)$ without issue.

We start with the nonsimultaneous case.  The result is stated in Theorem \ref{teo.tasas.mosimultena}.
\begin{proof}[Proof of Theorem \ref{teo.tasas.mosimultena}]
    Consider $(u,v)$ a quenching solution of \eqref{1.1} defined up to time $T$. Assume that $u$ is the quenching component (the other case is similar), that is,
$$
\lim_{t\to T} u(x_u(t),t)=0, \qquad
v(x,t)\ge \delta \text{   for } (x,t) \in \overline\Omega \times [0,T).
$$
Therefore, at point $(x_u(t),t)$ we have that the diffusion is bounded as always and the absorption is unbounded, then
\begin{equation} \label{ut.bounds}
u_t (x_u(t),t) =J*u (x_u(t),t) -u (x_u(t),t)-\lambda u^{-\alpha}v^{-p}(x_u(t),t) \sim - u^{-\alpha}(x_u(t),t)
\end{equation}
and the quenching rate is obtained by integration. The fact that $q-\alpha<1$ follows from Corollary \ref{lema.auxiliar.nosim}.
\end{proof}

For the simultaneous case, stated in Theorem \ref{lema.qrate.sim}, we will need some relation between $u$ and $v$. 
As shown earlier, Proposition \ref{prop.des.clave} and Corollary \ref{lema-estimaciones} gives us this relation in the case $p-\beta,q-\alpha\ge1$. However, as said in Remark \ref{proposicion vacia}, in the case $p-\beta,q-\alpha<1$, Proposition  \ref{prop.des.clave} is empty and we need to impose the extra hypothesis $x_u(t)=x_v(t)$.

\begin{proof}[Proof of Theorem \ref{lema.qrate.sim}]
Let $(u,v)$ be a simultaneous quenching solution, then 
$$
\liminf_{t\to T} u(x_u(t),t)=0, \qquad
\liminf_{t\to T} v(x_v(t),t)=0.
$$
In fact, thanks to \eqref{eq-m} we have 
\begin{equation} \label{todoacero}
        \lim_{t\nearrow T} u(x_u(t),t) =  \lim_{t\nearrow T} v(x_v(t),t) = 0.
\end{equation}
Indeed, suppose this is not true and  $\limsup_{t\to T} u(x_u(t),t)=c>0$. Since, for $t$ close to $T$, the function $u$ oscillates between $(0,c)$, there exists a sequence of times $t_n\to T$ such that
    $$
    u(x_u(t_n),t_n)= c/2,\qquad
    u_t(x_u(t_n),t_n)\ge0.
    $$
    
In the case $p-\beta, q-\alpha \geq 1$, we note that, since $u(x_v(t_n),t_n)\ge u(x_u(t_n),t_n) =  c/2$ for every $n\in\mathbb{N}$, Lemma \ref{max(p,q)>=1} gives us that $v(x_v(t_n),t_n) \geq  \widetilde c>0$. Therefore,
$$
\min \left\{ \min_{x\in\overline\Omega} u(x,t_n), \min_{x\in\overline\Omega} v(x,t_n)\right\} \ge \min\left\{\frac{c}{2},\widetilde c\right\}>0,
$$
for every $n\in\mathbb{N}$, which is a contradiction with \eqref{eq-m}.

On the other hand, if $p-\beta,q-\alpha< 1$, we assume that $x_u(t)=x_v(t)$. Then from \eqref{eq-m} we have that $v(x_u(t_n),t_n)\to 0$ which implies $u_t(x_u(t_n),t_n)\to-\infty$, and we get another contradiction.

Now observe that the diffusion is always bounded while the absorption terms in this case are unbounded. Therefore, there exists a time $t_0 \in [0,T)$ such that
\begin{equation} \label{equiv.ut}
    u_t(x_u(t),t) \sim - v^{-p} u^{-\alpha} (x_u(t),t), \qquad
    v_t(x_v(t),t) \sim - u^{-q} v^{-\beta} (x_v(t),t)
\end{equation}
for every $t\in[t_0,T)$. Which is a crucial estimate to obtain the quenching rates. We consider several cases in terms of the parameters.

 $i)$ Assume that $p-\beta,q-\alpha > 1$. From Corollary \ref{lema-estimaciones} we know that there exists $t_1\in [t_0,T)$ such that
\begin{equation}\label{equiv.uv1}
    u^{1-q+\alpha} (x_u (t),t) \sim v^{1-p+\beta} (x_u (t),t),\qquad
    u^{1-q+\alpha} (x_v (t),t) \sim v^{1-p+\beta} (x_v (t),t)
\end{equation}
for every $t\in[t_1,T)$. Using the first equivalence in \eqref{equiv.uv1} along with the first one in \eqref{equiv.ut} we get
\begin{equation*}
    u_t(x_u(t),t) \sim - u^{\frac{pq-p-\alpha-\alpha\beta}{1-p+\beta}}(x_u(t),t).
\end{equation*}
Then we can integrate this equivalence between $t\in[t_1,T)$ and $T$ and arrive to the desired quenching rate
\begin{equation*}
    u(x_u(t),t) \sim (T-t)^{\frac{p-1-\beta}{pq-(1+\alpha)(1+\beta)}}.
\end{equation*}
The quenching rate  for the component $v$ is given by the second equivalence in \eqref{equiv.uv1}.
\begin{equation*}
    v(x_u(t),t) \sim (T-t)^{\frac{q-1-\alpha}{pq-(1+\alpha)(1+\beta)}}.
\end{equation*}

To consider the behaviour of the solution along the sequence $(x_v(t),t)$ we observe that, on one hand,
$$
v(x_v(t),t)\le v(x_u(t),t)\le C (T-t)^{\frac{q-1-\alpha}{pq-(1+\alpha)(1+\beta)}}.
$$
And on the other hand,
$$
C (T-t)^{\frac{p-1-\beta}{pq-(1+\alpha)(1+\beta)}}\le u(x_u(t),t) \le u(x_v(t),t)\le C v^{\frac{p-1-\beta}{q-1-\alpha}}(x_v(t),t).
$$
As before, the quenching rate  for the component $u$ is given by \eqref{equiv.uv1}.

$ii)$ Assume now that $p-\beta> 1= q-\alpha$ and consider the behaviour of the solution along the sequence $(x_u(t),t)$ (the quenching rate along $(x_u(t),t)$ gives us the quenching rate along $(x_v(t),t)$ with the same reasoning as in case $i)$).

In this case, Corollary \ref{lema-estimaciones} gives us that there exists $t_2\in [t_0,T)$ such that
\begin{equation}\label{equiv.uv2}
    v^{1-p+\beta} (x_u(t),t) \sim -\log u(x_u(t),t) = \log \left(\frac{1}{u(x_u(t),t)} \right)
\end{equation}
for every $t\in[t_2,T)$.  Then by \eqref{equiv.ut},  
\begin{equation}\label{eq.log}
  u_t(x_u(t),t) \sim - u^{-\alpha}(x_u(t),t)\left(\log \left( \frac1{u(x_u(t),t)} \right)\right)^{\frac{p}{p-1-\beta}}.
\end{equation}
Integrating this expression between $t\in[t_2,T)$ and $T$, the rate is given implicitly by
\begin{equation} \label{integral.rate}
\int_{u(x_u(t),t)}^0 s^{\alpha}\left( \log \left( \frac1{s} \right) \right)^{\frac{p}{1-p+\beta}}\,ds\sim -(T-t).
\end{equation}
We can then apply the L'H\^opital rule over the following limit to get
\begin{equation}\label{limitelog}
    \lim_{t\nearrow T} \frac{\int^{u(x_u(t),t)}_{0} s^{\alpha} \log (1/s)^{\frac{p}{1-p+\beta}} ds}{ u^{1+\alpha}(x_u(t),t) \log (1/u(x_u(t),t))^{\frac{p}{1-p+\beta}}}=\frac{1}{1+\alpha}.
\end{equation}
Therefore, using both \eqref{integral.rate} and \eqref{limitelog}, there exists $t_3\in[t_2,T)$ such that
$$
\log \left(\frac{1}{u(x_u(t),t)}\right)^{\frac{p}{p-1-\beta}}\sim \frac{u^{1+\alpha}(x_u(t),t)}{(T-t)}
$$
for every $t\in[t_3,T)$. Plugging these estimates in \eqref{eq.log}
we can integrate between $t_3$ and $t \in [t_3,T)$ to obtain
\begin{equation} \label{equiv.logu.logT}
    \log u(x_u(t),t) \sim \log(T-t),
\end{equation}
for every $t\in[t_3,T)$. Using again this estimate in \eqref{eq.log} we get
$$
u^{1+\alpha}(x_u(t),t)\sim \int_t^T (-\log(T-s))^{\frac{-p}{1-p-\beta}} ds
= \int_{-\log(T-t)}^\infty \tau^{\frac{-p}{1-p-\beta}} e^{-\tau} d\tau.
$$
That is, $u^{1+\alpha}(x_u(t),t)$ behaves like the upper incomplete Gamma function with the following parameters:
\begin{equation*}
    u^{1+\alpha}(x_u(t),t) \sim \Gamma\left( \frac{1-2p-\beta}{1-p-\beta}, -\log (T-t) \right).
\end{equation*}
Using the asymptotic properties of the upper incomplete Gamma function, we know that:
\begin{equation*}
    \frac{\Gamma\left( \frac{1-2p-\beta}{1-p-\beta}, -\log (T-t) \right)}{(-\log (T-t))^{-p/(1-p-\beta)}(T-t)} \xrightarrow{t\rightarrow T} 1,
\end{equation*}
and the quenching rate for $u$ follows. The quenching rate for the $v$ component follows immediately from
\eqref{equiv.logu.logT} and  \eqref{equiv.uv2}.

$iii)$ The proof for this case is analogous to that of point $ii)$.

$iv)$ The proof for this case is analogous to that of point $i)$.

$v)$ In this case both components reach the minimum at the same point, that is $x(t):=x_u(t)=x_v(t)$. Then,
    $$
    \lim_{t\nearrow T} \frac{v^p u^\alpha u_t}{u^q v^\beta v_t}(x(t),t) =
    \lim_{t\nearrow T} \frac{v^p u^\alpha (J*u-u)-\lambda}{u^q v^\beta (J*v-v)-\mu}(x(t),t) =
    \frac{\lambda}{\mu}.
    $$
    This implies that there exists $t_2\in[t_0,T)$ such that
    $$
    u^{\alpha-q} u_t(x(t),t)\sim v^{\beta-p} v_t(x(t),t).
    $$
    Integrating between $t\in[t_2,T)$ and $T$ we get the estimate
    $$
    u^{1+\alpha-q}(x(t),t)\sim v^{1+\beta-p}(x(t),t).
    $$
    At this point we can follow the same argument as in case $i)$ to obtain the quenching rates.
\end{proof}

\end{section}

\centerline{{\bf Acknowledgements}}

S. Junquera is  partially supported by grant  PID2022-137074NBI00 funded by \linebreak MCIU/AEI/10.13039/501100011033, Spain and by the FPU21/05580 grant from the Ministry of Science, Innovation and Universities.

\end{document}